\documentclass[10pt]{amsart}
\usepackage{amsfonts,amsmath,amssymb,amsthm,latexsym,verbatim,url}
\newtheorem{theorem}{Theorem}[section]
\newtheorem{lemma}[theorem]{Lemma}

\newtheorem{corollary}[theorem]{Corollary}
\theoremstyle{remark}

\theoremstyle{definition}

\DeclareMathOperator{\Gal}{{\mathrm{Gal}}}
\DeclareMathOperator{\Aut}{{\mathrm{Aut}}}
\newcommand{\abs}[1]{|#1|}	
\newcommand{\F}{{\mathbb {F}}} 
\newcommand{\R}{{\mathbb {R}}} 
\newcommand{\Z}{{\mathbb {Z}}} 
\newcommand{\Q}{{\mathbb {Q}}} 
\newcommand{\C}{{\mathbb {C}}}
 
\newcommand{\Cay}{{\mathrm {Cay}}}
\newcommand{\GL}{{\mathrm {GL}}} 
\newcommand{\om}{{{\omega}}}

\begin{document}
\title[]{Large sets of strongly cospectral vertices in Cayley graphs}

\author{Peter Sin}
\address{Department of Mathematics\\University of Florida\\ P. O. Box 118105\\ Gainesville FL 32611\\ USA}
\thanks{This work was partially supported by a grant from the Simons Foundation (\#633214
 to Peter Sin)}
\thanks{}

\begin{abstract} Strong cospectrality is an equivalence relation on the
  set of vertices of a graph that is of importance in the
  study of quantum state transfer in graphs.  We construct  families of abelian Cayley
  graphs  in which the number of  mutually strongly cospectral vertices can be arbitrarily
  large.
\end{abstract}

\maketitle

\section{Introduction} Let $X$ be a simple graph with  adjacency matrix $A$, for some
fixed ordering of the vertex set $V(X)$.
For each  $a\in V(X)$, let $e_a$ denote its characteristic vector, the column
vector in $\C^{V(X)}$ with 1 in the $a$ entry and zero elsewhere.
We consider the spectral decomposition of $A$,
\begin{equation}
  A=\sum_{r=1}^k \theta_rE_r,
\end{equation}
where $\theta_1$,\ldots, $\theta_k$ are the distinct eigenvalues of $A$ and $E_r$
is the idempotent projector onto the $\theta_r$ eigenspace.

Two vertices $a$ and $b$ are said to be {\it strongly cospectral} if and only if
for all $r$ we have $E_re_a=\pm E_re_b$.
The terminology is justified  by the fact that the above condition implies that
$(E_r)_{a,a}=(E_r)_{b,b}$ for all $r$, which is one of several equivalent definitions of {\it cospectrality} of $a$ and $b$.

The significance of strong cospectrality in the study of quantum state transfer was first observed by Godsil in \cite{G12}.  A continuous time quantum walk on the graph $X$ is defined by the family of unitary matrices $U(t)=e^{-itA}$, $t\in \R$. We say that there is {\it perfect state transfer}
from $a$ to $b$ if, for some $t_0$, we have  $\abs{U(t_0)_{b,a}}=1$. It was shown \cite[2.2 Corollary]{G12} that if there is perfect state transfer from $a$ to $b$, then the two vertices must be strongly cospectral. Kay (\cite[IV.D]{K}) showed that if there is
perfect state transfer from $a$ to $b$ and from $a$ to $c$, then $b=c$.
It is not hard to see that the existence of perfect state transfer between vertices
defines an equivalence relation, and Kay's result means that the equivalence classes
contain at most two elements. For the weaker equivalence relation of strong cospectrality
small examples show that the equivalence classes can have size greater than 2.
For example, the cartesian product $P_2\square P_3$ of paths of lengths 2 and 3 has  a strong cospectrality class of 4 elements.
So a natural question, which appears as Problem 9 in Coutinho's thesis
\cite{CPhD}, is: What is the maximum size of a set of mutually strongly cospectral vertices of a graph?
In a multiple cartesian product of paths, if the path lengths are chosen suitably so that the product has simple eigenvalues, the ``corners'' form a large set of mutually strongly cospectral vertices,
which shows that strongly cospectral sets can be arbitrarily large in general. However,
the same question for vertex-transitive graphs has remained open.
Work on this latter question has so far focused on obtaining  upper bounds on the size of a strong cospectrality class in terms of other data, such as
 the maximum eigenvalue multiplicity \cite[Theorem 6.1]{AG}.
 We shall show that no absolute bound exists, by exhibiting Cayley graphs with arbitrarily large sets of mutually strongly cospectral vertices.

Our constructions are of abelian Cayley graphs. In \S\ref{Cayley} we shall review some general theory of Cayley graphs in relation to strong cospectrality, and state
the group-theoretic formulation of  strong cospectrality that we shall use exclusively. 
In \S\ref{hetero} we consider cartesian products of certain Cayley graphs of cyclic
groups whose orders are distinct powers of 2, each $\ge8$. By applying Galois theory
of the fields of 2-power roots of unity, we show that every involution
is strongly cospectral with the zero vertex. As shown in \cite[Corollary 3.3]{AG}, the set of vertices in any Cayley graph that are strongly cospectral with the identity element forms a subgroup, which we shall call the {\it strongly cospectral subgroup}.  Moreover, it is not difficult to see that elements that are strongly cospectral with the identity must be involutions, so this subgroup is an elementary abelian 2-group.
(See  the proof of \cite[Lemma 4.1]{AG} .)
Thus, in our examples, the strongly cospectral subgroup is as large as possible
and it can be arbitrarily large.

In \S\ref{cubelike} we construct cubelike graphs with
arbitrarily large strongly cospectral subgroups. Here we view the underlying elementary
abelian group as a vector space over the field $\F_2$ of two elements
so that we can apply some results on the geometry of quadratic forms over $\F_2$.
These graphs provide examples of arbitrarily large strongly cospectral classes
in graphs with integral eigenvalues.

In both constructions, in order to obtain a large strongly cospectral
subgroup, the whole group has to be much larger. This accords with earlier results such as \cite[Corollary 8.3]{AG}, which shows that in a cubelike graph the order of
the strongly cospectral subgroup is bounded by the square root of the group order.

Our examples raise some questions about {\it pretty good state transfer}, a condition on vertices which is intermediate in strength between perfect state transfer and strong cospectrality. We comment on these questions briefly in the final section. 

\section{Strongly cospectral vertices in Cayley graphs}\label{Cayley}
Let $G$ be a finite group and $S$ a subset of $G$. We shall always assume
that $S$ is closed under inversion and does not contain the identity element.
Then the Cayley graph $\Cay(G,S)$ with vertex set $G$
and connection set $S$ is a simple graph. 

By a {\it normal} Cayley graph we mean one in which
the connection set is a union of conjugacy classes.
The eigenvalues of a normal Cayley graph  $\Cay(G, S)$
are given by the irreducible  characters of $G$.
For any character $\chi$ of $G$ set $\chi(S):=\sum_{s\in S}\chi(s)$.
Then the eigenvalues  are the values
$\chi(S)/\chi(1)$, as $\chi$ varies over the irreducible characters of $G$, with
each $\chi$ contributing $\chi(1)^2$ to the total multiplicity of the eigenvalue.

As mentioned in the Introduction, the vertices in any Cayley graph 
that are strongly cospectral with the identity element form a subgroup,
called the {\it strongly cospectral subgroup}.
In a normal Cayley graph, this subgroup is a central
elementary abelian $2$-subgroup, by  \cite[Lemma 4.1]{AG}.

An important tool for us is  the following characterization  \cite[Theorem 2.4]{SS} of strong cospectrality in normal Cayley graphs. We shall need only the abelian case,
which is well known (See \cite[16.2]{CGbook} for example.) 

\begin{lemma}\label{strong} In a normal Cayley graph $\Cay(G,S)$, two vertices $g$ and $h$
  are strongly cospectral if and only if the following hold.
  \begin{enumerate}
  \item[(a)] $h=zg$ for some central involution $z$;
  \item[(b)] If $\chi$ and $\psi$ are  irreducible characters such that
    $\chi(S)/\chi(1)=\psi(S)/\psi(1)$, then $\chi(z)/\chi(1)=\psi(z)/\psi(1)$.
  \end{enumerate}
\end{lemma}
\qed

\section{Some Cayley graphs in heterocyclic groups}\label{hetero}
In this section we shall construct Cayley graphs with large strongly cospectral
subgroups. The graphs will be cartesian products of Cayley graphs on cyclic
groups of different orders. We begin by recalling some facts about
characters of cyclic groups. Then we shall define the cyclic Cayley graphs
of interest and examine first their eigenvalues, then those of their
cartesian products.

Let $Z_d$ denote a cyclic group of order $d$, written mutiplicatively.
We use the notation $C_m=\langle x_m\rangle$ for $Z_{2^m}$. We shall assume that $m\ge 3$.
Let $\om_m=\exp(\frac{2\pi i}{2^m})$, a primitive $2^m$-th root of unity in $\C$.
We identify the group $C_m^\vee$ of irreducible complex characters with
$\Z/2^m\Z$, where $a\in\Z/2^m\Z$ corresponds the the character $[a]: x_m\mapsto \om_m^a$.

We shall consider the fields $\Q(\om_m)$ and their subfields $F_m=\Q(\om_m+\om_m^{-1})$.
The following lemma summarizes some well known facts from Galois theory that
we shall need.
\begin{lemma}\label{galois}
  \begin{enumerate}
 \item[(a)]  $\Gal(\Q(\om_m)/\Q)=\langle\beta_m\rangle\times\langle\gamma_m\rangle\cong Z_2\times Z_{2^{m-2}}$,  where $\beta_m(\om_m)=\om_m^{-1}$ and $\gamma_m(\om_m)=\om_m^5$.
 \item[(b)] Let $\tau_m$ be the unique involution of $\langle\gamma_m\rangle$. Then
 $\tau_m(\om_m)=-\om_m$.
 \item[(c)] The restriction map $\Gal(\Q(\om_m)/\Q)\to \Gal(F_m/\Q)$  defines an
 isomorphism of $\langle\gamma_m\rangle$ with  $\Gal(F_m/\Q)$.
 \item[(d)] The field $F_{m-1}$ is the subfield of $F_m$ fixed by $\tau_m$.
  \end{enumerate}
\end{lemma}
\begin{proof} For the convenience of the reader we shall indicate a proof of (a). The other parts follow from (a) by standard Galois theory. Each element of the
  Galois group  $\Gal(\Q(\om_m)/\Q)$ defines an automorphism of the cyclic group
  $\langle \om_m\rangle$, so we have an injective group homomorphism
  $\Gal(\Q(\om_m)/\Q)\to\Aut(\langle \om_m\rangle)$. It is well known
  that $\Gal(\Q(\om_m)/\Q)$ and $\Aut(\langle \om_m\rangle)$ both have order
  $\phi(2^m)=2^{m-1}$, where $\phi$ is Euler's function, so the homomorphism
  is an isomorphism. The precise structure of the  automorphism group of a cyclic group of order $2^m$ can be found in group theory textbooks, for example \cite[23.3]{Asch}.
\end{proof}

In the group $C_m$, we consider the generating set
\begin{equation}\label{Tmdef}
T_m=\{x_m^{2i+1}\mid 0\leq i\leq 2^{m-3}-1\}\cup\{x_m^{-(2i+1)} \mid 0\leq i\leq 2^{m-3}-1\}.
\end{equation}

We shall use the following notation. For any subset $S$ of a group $G$ and any character $\lambda$ of $G$, we shall write $\lambda(S)$ to mean $\sum_{s\in S}\lambda(s)$.
Since $T_m$ is closed under inversion, it follows that $[a](T_m)\in F_m$ for all $a\in\Z/2^m\Z$.  

\begin{lemma}\label{even-odd}
  \begin{enumerate}
    \item[(a)] If $a$ is odd, then $[a](T_m)\neq 0$.
    \item[(b)] $\tau_m([a](T_m))=-[a](T_m)$  if $a$  is odd.
    \item[(c)]$[a](T_m) \in \Q$ if $a$ is even.
  \end{enumerate}
  \end{lemma}
\begin{proof}
  First observe that $[1](T_m)\neq 0$ as it is  the sum of positive real numbers $\om^i+\om^{-i}$. Then (a) holds because $[a](T_m)$ is a Galois conjugate of $[1](T_m)$.
    If $a$ is odd, then $[a](T_m)$ is a sum of odd powers of $\om_m$, so (b)
    follows from Lemma~\ref{galois}(b). If $a$ is even then we can write $a=2r$, with $0\leq r\leq 2^{m-1}-1$.
  Starting from the fact that $\om_m^2=\om_{m-1}$, it follows that
  $[2r](x_m)=\om_m^{2r}=\om_{m-1}^r$. From this, and the definition of $T_m$ we obtain
  \begin{equation}
      [2r](T_m)=\sum_{i=0}^{2^{m-3}-1} \om_{m-1}^{r(2i+1)}+\om_{m-1}^{-r(2i+1)}
      = \sum_{k=0}^{2^{m-2}-1}\om_{m-1}^{r(2k+1)}.
  \end{equation}
  The last sum is equal to the sum of values of the character $[r\pmod{2^{m-1}}]$
  of $C_{m-1}$, summed over the nontrivial coset of the unique subgroup
   $\langle x_{m-1}^2\rangle$ of index 2 in $C_{m-1}$.
    Therefore, we have
    \begin{equation}
    [2r](T_m)= \begin{cases}  2^{m-2}  \qquad \text{if $r=0$}\\
          -2^{m-2}  \qquad \text{if $r=2^{m-2}$}\\
          0, \qquad \text{if $r\notin\{0, 2^{m-2}\}$},
        \end{cases}
    \end{equation}
    which completes the proof of  (c).
\end{proof}

Let $J$ be a finite set of positive integers $j\ge 3$.
We shall be interested in the groups  $G_J=\bigoplus_{j\in J}C_j$. We identify the direct factors $C_j$ with their canonical images in  $G_J$ in the usual way. Then
$S_J=\cup_{j\in J} T_j$ is a generating set of $G_J$. The graph $\Cay(G_J, S_J)$
is the cartesian product $\square_{j\in J}\Cay(C_j,T_j)$. Our goal is to show, using Lemma~\ref{strong}, that the strongly cospectral subgroup of $\Cay(G_J, S_J)$ is the subgroup
of all elements whose square is the identity.

The characters of $G_J$ are given by tuples $a_J=([a_j])_{j\in J}$ with $a_j\in\Z/2^j\Z$. We have
\begin{equation}
  a_J(S_J)=\sum_{j\in J}[a_j](T_j).
\end{equation}

\begin{lemma}\label{mainlemma} Let $J$ be a finite set of positive integers $j\ge 3$.
  Let $a_J=([a_j])_{j\in J}$ and $b_J=([b_j])_{j\in J}$ be characters of $G_J$, and suppose that $a_J(S_J)=b_J(S_J)$. Assume that not all $a_j$ and $b_j$ are even
  and let $m\in J$ be the largest element for which
  either $a_m$ or $b_m$ is odd. Then both $a_m$ and $b_m$ are odd and
  $[a_m](T_m)=[b_m](T_m)$.
\end{lemma}
\begin{proof} Without loss of generality we may assume that $a_m$ is odd.
  By Lemma~\ref{even-odd}  $[a_j](T_j)$ and $[b_j](T_j)$
  are rational for $j>m$. For $j<m$, the values $[a_j](T_j)$ and $[b_j](T_j)$ lie in $F_{m-1}$. Let
  \begin{equation*}
  y=\sum_{\begin{smallmatrix}j\in J\\
      j\neq m
    \end{smallmatrix}} ([a_j](T_j)-[b_j](T_j)).
  \end{equation*}
  The equation $a_J(S_J)=b_J(S_J)$ is equivalent to
  \begin{equation}\label{yeq}
    y=[b_m](T_m)-[a_m](T_m).
  \end{equation}
Now  $y\in F_{m-1}$, the fixed field
  of $\tau_m$. If $b_m$ is odd  Lemma~\ref{even-odd} forces $y=0$, as the right
  side of (\ref{yeq})  is negated by $\tau_m$. If $b_m$ is even, we simply write
  (\ref{yeq}) as $[b_m](T_m)-y=[a_m](T_m)$ and obtain a contradiction, as $\tau$ fixes
  the left side of the latter equation and negates the right, by Lemma~\ref{even-odd}.
\end{proof}

\begin{corollary}\label{charsagree}Let $a_J=([a_j])_{j\in J}$ and $b_J=([b_j])_{j\in J}$ be characters of $G_J$, and suppose that $a_J(S_J)=b_J(S_J)$. Then for every $j\in J$,
  $a_j$ and $b_j$ are either both odd or both even. In particular $a_J(t)=b_J(t)$
  for every involution $t\in G_J$.
\end{corollary}
\begin{proof} The proof is by induction on $\abs{J}$. If $J=\{j\}$, then $G_J$
  is a cyclic group and the statement follows from Lemma~\ref{even-odd}
  (using the fact that $j\ge 3$). Suppose $\abs{J}>1$ and the statement
  holds for smaller $J$. If all $a_j$ and $b_j$ are even
  there is nothing to prove. Otherwise, we can apply Lemma~\ref{mainlemma} to the
  largest element $m\in J$ for which either $a_m$ or $b_m$ is odd,
  and then cancel $[a_m](T_m)=[b_m](T_m)$ from the equation
  $a_J(S_J)=b_J(S_J)$, to get $a_{J'}(S_{J'})=b_{J'}(S_{J'})$, where
  $J'=J\setminus\{m\}$. The statement now follows by induction.
\end{proof}

Our results can be summarized in the following theorem.
\begin{theorem} Let $J$ be a finite set of positive integers $j\ge 3$.
  For each $j\in J$ let $C_j$ be a cyclic group of order $2^j$ and
  let $T_j$ be the subset of $C_j$ defined in \eqref{Tmdef}. Then in the Cayley graph $\Cay(G_J,S_J)=\square_{j\in J}\Cay(C_j, T_j)$  every  element whose square is the identity belongs to the strongly cospectral subgroup.
  Hence the strongly cospectral subgroup has order $2^{\abs{J}}$.
\end{theorem}
\begin{proof}
  By Corollary~\ref{charsagree} and Lemma~\ref{strong}, all involutions of
  $\Cay(G_J,S_J)$ belong  to the strongly cospectral subgroup. On the other hand, all nonidentity elements  of this subgroup must be involutions, by Lemma~\ref{strong}(a).
\end{proof}

\section{Strongly cospectral subgroups in cubelike graphs}\label{cubelike}
A cubelike graph is defined as a Cayley graph $\Cay(G,S)$ where $G$ is an elementary abelian $2$-group and the connecting set  $S$ is any subset of $G$ that does not contain
the identity element. Let $\abs{G}=2^n$. We will identify $G$ with the additive group of the vector space $\mathbb F_2^n$ over $\mathbb F_2$, so we may speak of hyperplanes instead of subgroups of index $2$, make use of the dot product and, later on, quadratic forms.

In this section we shall construct cubelike graphs with large strongly cospectral subgroups. The graphs will be cartesian products of smaller cubelike graphs, with each factor
contributing (at least) one dimension to the $\mathbb F_2$-dimension of the strongly
cospectral group. We first review characters of cubelike graphs, from a geometric
point of view. Then, to illustrate the main idea, we look at  a simple example with only three cartesian factors and complete connecting sets. While this initial attempt does not
produce a large strongly cospectral subgroup,  we modify the construction to make it work by using connecting sets defined by nondegenerate quadratic forms in odd dimensions.

\subsection{Characters} 
For $w\in G$, $\chi_w:G\to\{\pm1\}$ is the character defined by
$\chi_w(x)=(-1)^{w\cdot x}$.
Then the eigenvalues of $\Cay(G,S)$, counted with multiplicity, are the $2^n$ values
$$
\chi_w(S):=\sum_{s\in S}\chi_w(S)=\sum_{s\in S}(-1)^{w\cdot s}=
\abs{S}-2n_w,
$$
where $n_w=\abs{\{s\in S\mid w\cdot s=1\}}$.

We have $n_w=\abs{S}-\abs{S\cap H_w}$ where, for $w\neq 0$, $H_w$ is the hyperplane
orthogonal to $w$ with respect to the dot product, and $H_0=G$.

By Lemma~\ref{strong} an element $z$ belongs to the strongly cospectral  subgroup if and only if any two characters $\chi_w$ and $\chi_u$  that give the same
eigenvalue $\chi_w(S)=\chi_u(S)$  satisfy $\chi_w(z)=\chi_u(z)$.

Let $\sigma=\sum_{s\in S}s$, where we mean the sum in the group $G$.
Then
$$
\chi_w(\sigma)=\prod_{s\in S}(-1)^{w\cdot s}=(-1)^{n_w}.
$$
In particular, we see that for any $w$, the eigenvalue
$\chi_w(S)$ determines $n_w$, which determines $\chi_w(\sigma)$. 
Thus, the element $\sigma$ belongs to the strongly cospectral subgroup,
This idea can be traced back to \cite{BGS}.

\subsection{Plan for construction}\label{idea}
We begin by outlining the general strategy for extending the above idea to
create more elements in the strongly cospectral subgroup. 
For the purpose of illustration,
suppose $S=S_1\cup S_2\cup S_3$ is the disjoint union of three
subsets. Let $\sigma_i$ be the sum in $G$ of the elements of $S_i$.
Now, suppose that for all $w\in G$,  $\abs{H_w\cap S}$ determines
$\abs{H_w\cap S_i}$ for all $i$.
Then $\chi_w(S)$ determines $n_w$, which determines $\abs{H_w\cap S}$,
which determines $\abs{H_w\cap S_i}$ for all $i$, which determines
$\chi_w(\sigma_i)$. Thus, by Lemma~\ref{strong}, the elements $\sigma_i$ are cospectral to $0$ (although they may be equal to zero in some cases).
If we can find $S$ and $S_i$ as above such that
the $\sigma_i$ generate a group of order 8, we will have
an example of a strongly cospectral subgroup of order 8.
Of course, we can try to generalize to $k$ subsets $S_i$ to construct
a strongly cospectral subgroup of order $2^k$.

We next consider a prototype for this idea (in which unfortunately the $\sigma_i$ are equal to $0$).
Suppose we write $n=n_1 +n_2 +n_3$, with $n_1 \ll n_2 \ll n_3$.
Then accordingly, we can decompose $G$ as
$G= V_1\oplus V_2\oplus V_3$, with $\dim_{\mathbb F_2}V_i=n_i$.
We view this direct sum internally, so $V_2$ is the set of triples 
$(0,v,0)$, with $v\in\mathbb F_2^{n_2}$.
Let $S_i=V_i\setminus\{0\}$ and $S=\cup_{i=1}^3 S_i$.
The graph $\Cay(G, S)$ is the cartesian product $\square_{i=1}^3\Cay(V_i,S_i)$.
Then for $w\in G$, we have $H_w\cap S_i=2^{n_i}-1$ or $2^{n_i-1}-1$ depending on whether $V_i\subseteq H_w$ or not. Hence, as $w$ varies, there are 8 possibilites for the sequence
$\{\abs{H_w\cap S_i}\}_{i=1}^3$. As long as the $n_i$ are chosen properly, say with $n_3$ very large, $n_2$ moderate and $n_1$ small, we can tell from
$\abs{H_w\cap S}$ which of the 8 sequences we have. That is to say, $\abs{H_w\cap S}$ determines the $\abs{H_w\cap S_i}$.

In this prototype, since $S_i$ is permuted by  $\GL(V_i)$,
the sum $\sigma_i$ of its elements is fixed by $\GL(V_i)$, so  $\sigma_i=0$ and
we do not have a working construction yet.
We see then that in order to make the construction work
we will need to choose the connection set $S_i$ of each
cartesian factor so that its setwise stabilizer in $\GL(V_i)$
has a nonzero fixed point. At the same time we wish to preserve the crucial property
of our prototype that for each $H_w$, the intersection
sizes $\abs{H_w\cap S_i}$ are all determined by $\abs{H_w\cap S}$. 
The solution is provided by quadratic forms over $\mathbb F_2$.

\subsection{Quadratic forms over $\F_2$}
We refer the reader to \cite{Kap} and \cite{Dye} for the basic theory of  quadratic forms
over $\F_2$, but we also point out that many facts that we use can be
verified by direct computation. For example, the zeros of a form given in coordinates
are easy to count, as the form takes only two values.
Let $V=\mathbb F_2^d$, where $d=2e+1$, $e\ge 1$.
On $V$ we take coordinates $x_1$,\ldots, $x_d$ and consider the
quadratic form
\begin{equation}\label{qform}
  q(x_1,\ldots,x_d)= x_d^2+ \sum_{i=1}^e x_ix_{e+i}.
\end{equation}
The bilinear form $b(v,v')=q(v+v')-q(v)-q(v')$ associated with $q$ has a 1-dimensional
radical $\langle p\rangle$, where $p=(0,\ldots,0,1)$, called the {\it nucleus} of $q$.
Note that $q(p)=1$.

A subspace is called {\it totally isotropic} with respect to a quadratic form
if the restriction of the form to the subspace vanishes entirely.
For nondegenerate quadratic forms over $\F_2$ in even dimension $2m$ the dimension of
a maximal totally isotropic subspace (also called the Witt index)
is either $m$ or $m-1$ \cite[Theorem 27, p. 33]{Kap}.
The forms with Witt index $m$ are called {\it hyperbolic}
and those with Witt index $m-1$ are called {\it elliptic}.
Hyperbolic and elliptic forms are also said to have Arf invariant 0 and 1 respectively.
It will be useful for us to consider the subspace $W$ defined by $x_d=0$
and the restriction of the form $q$ to $W$, which we shall denote by $f$.
If we use the same notation $x_i$ for the coordinate restricted to $W$,
then $f(x_1,\ldots,x_{2e})=\sum_{i=1}^e x_ix_{e+i}$.
The form $f$ is hyperbolic, as the subspace defined by the vanishing of the
first $e$ coordinates is totally isotropic.

Let $Q$ be the set of  zeros of $q$ in $V\setminus\{0\}$.
We shall compute the sum $\sigma_Q=\sum_{v\in Q}v$ in $V$.
  Let vectors in $V$ be written as $(v,a)$ with first component $v\in W$
  and second component $a\in\mathbb F_2$. Then $(v,a)$  belongs to $Q$ iff $f(v)=a$.

\begin{lemma}\label{sumQ} 
  If $d=3$, then  $\sigma_Q=(0,0,1)$ and for $d\ge 5$ we have $\sigma_Q=0$ . 
\end{lemma}
\begin{proof} The result for $d=3$ can be seen by a short computation, so we assume $d\ge 5$ or, equivalently, $e\ge 2$. We have $\sigma_Q=\sum_{v\in W}(v,f(v))$.  
  Certainly, $\sum_{v\in W}v=0$ and, since  the
  number of nonzeros of the hyperbolic form $f$ is $2^{2e-1}-2^{e-1}$  \cite[Lemma 9.4.1]{BCN}, which is even,  we also have  $\sum_{v\in W}f(v)=0$. Thus $\sigma_Q=0$. 
  \end{proof}

By the Lemma, if $d\ge 5$ and we set $S'=Q\cup\{p\}$, then
$\sigma_{S'}:=\sigma_Q+p=p\neq 0$.

\subsection{Cubelike graphs from quadratic forms}\label{qgraphs}
We are now ready to  introduce our  cubelike graphs.
Let $k$ be a positive integer and 
  $G=V_1\oplus V_2\oplus\cdots\oplus V_k$, in which $V_i$ has odd dimension
  $n_i=2e_i+1$, with $n_i\ge 5$.
  In each $V_i$, we consider a  quadratic form as in \eqref{qform} (with $d=n_i$)
  and let $S_i=Q_i\cup\{p_i\}$ be the set consisting of the zeros of the form in
  $V_i\setminus\{0\}$ together with the nonzero vector in its nucleus.
  If we view $G$ as an internal direct sum and set  $S=\cup_{i=1}^kS_i$, then
  $\Cay(G,S)$ is the cartesian product $\square_{i=1}^k\Cay(V_i,S_i)$.
  Our  goal is to show that $\Cay(G,S)$ has a strongly cospectral group
of order $\ge 2^k$ if the dimensions $n_i$ are chosen appropriately.

 The next lemma determines the possible sizes of the
intersections of the sets $S_i$ with hyperplanes of $G$. 

\begin{lemma}\label{hyperplane_cuts}
  Let $V$ be a  subspace of $G=\F_2^n$, with $\dim V=d=2e+1\ge 5$.
  Let $Q$ be the set  zeros  in $V\setminus\{0\}$ of  a nondegenerate
quadratic form $q$ as given in \eqref{qform},
and let $p$ be the unique nonzero vector in its nucleus.
We set $S'=Q\cup\{p\}$. Let $H$ be either $G$ or a hyperplane of $G$. Then 
$\abs{H\cap S'}$ is given as follows.
\begin{enumerate}
\item[(i)] If $H$ contains $V$ then $\abs{H\cap S'}=\abs{S'}=2^{d-1}-1 +1=2^{d-1}$.

\item[(ii)] If $H$ does not contain $V$ but contains $p$, then $\abs{H\cap S'}=2(2^{d-3}-1) +1=2^{d-2}-1$.

\item[(iii)] If $H$ contains neither $V$ nor $p$ and the restriction of $q$ to $H\cap V$
is hyperbolic then $\abs{H\cap S'}=2^{d-2}+2^{e-1}-1$. 

\item[(iv)] If $H$ contains neither $V$ nor $p$ and the restriction of $q$ to $H\cap V$
  is elliptic  then  $\abs{H\cap S'}=2^{d-2}-2^{e-1}-1$.
  
\item[(v)] As $H$ and $H'$ vary over the set whose elements are $G$ and the hyperplanes
  of $G$  the minimum nonzero value of $\abs{\abs{H\cap S'}-\abs{H'\cap S'}}$ 
  is $2^{e-1}$.
\end{enumerate}

\end{lemma}
\begin{proof} (i) Follows from the fact that projection with respect to the nucleus
  maps $Q$ bijectively onto the set of nonzero vectors of $V/\langle p\rangle$.
  Suppose we are in case (ii). Then the restriction of
the associated symplectic form to $H\cap V$ has a 2-dimensional radical generated by
$p$ and $v$, say. Either $v$ or $p+v$ is isotropic for $q$, so we can assume that
$v$ was chosen to be isotropic. 
Then the the restriction of $q$ to $H\cap V$ induces  a nondegenerate quadratic
form on $(H\cap V)/\langle v\rangle$, analogous to $q$ but in dimension $d-2$.
We then find, using (i), that $\abs{H\cap S'}=2(2^{d-3}-1) +1=2^{d-2}-1$.
For parts (iii) and (iv), the restriction of the form $q$ to $H\cap V$ is
a nondegenerate quadratic form on the $2e$-dimensional space $H\cap V$.
The numbers stated in (iii) and (iv) are the numbers of
nonzero vectors on which such a form vanishes in the hyperbolic and elliptic cases, respectively \cite[Lemma 9.4.1]{BCN}. Part (v) follows from the previous parts.
\end{proof}

The following lemma, which applies more generally to a cartesian product
$\square_{i=1}^k\Cay(V_i,S_i)$ with arbitrary connecting sets $S_i$,
provides a sufficient condition for $\abs{H\cap S}$ to determine
all $\abs{H\cap S_i}$, whenever $H$ is either a hyperplane of $G$ or equal to $G$.

\begin{lemma}\label{capdetermine} Suppose $G=V_1\oplus V_2\oplus\cdots\oplus V_k$,
  where each $V_i$ is an $\F_2$-vector space. Let $n_i:=\dim V_i$.
  Let $S_i$ be any subset of $V_i\setminus\{0\}$ and $S=\cup_{i=1}^k S_i$. Let
    $N_i=\{\abs{H'\cap S_i} \mid \text{$H'$ a hyperplane of $G$ or $H'=G$}\}$ and  $\epsilon_i=\min\{\abs{a-b} \mid a, b\in N_i, a\neq b\}$. Suppose
    \begin{equation}\label{ebound}
      \epsilon_i\ge 2^{n_{i-1}+2}\quad\text{ for all $i=2$,\dots,$k$}.
    \end{equation}
    Then if $H$ is a hyperplane of $G$ or $H=G$,
    $\abs{H\cap S}$ determines $\abs{H\cap S_i}$ for all $i=1$,\dots,$k$.
  \end{lemma}
  \begin{proof} We use induction on $k$, the result being trivial for $k=1$.
    Assume $k>1$ and let $S'=\cup_{i=1}^{k-1}S_i$. Let $G'$ be the subgroup of $G$ consisting of elements whose  $k$-th component is zero.  Then $S'\subseteq G'$.
    Let $H$ be either equal to $G$ or a hyperplane of $G$. Then $H\cap G'$ is either
    equal to $G'$ or a hyperplane of $G'$, and $H\cap S'=(H\cap G')\cap S'$.
    Therefore, by the inductive hypothesis, $\abs{H\cap S'}$ determines $\abs{H\cap S_i}$ for all $i=1$,\dots, $k-1$. As $\abs{H\cap S}=\abs{H\cap S'}+\abs{H\cap S_k}$
    it suffices to show that $\abs{H\cap S}$ determines $\abs{H\cap S_k}$. The condition \eqref{ebound} implies that the $n_i$ form a strictly increasing sequence, and that 
    $$
    \abs{H\cap S}-\abs{H\cap S_k}=
    \sum_{i=1}^{k-1}\abs{H\cap S_i} < \sum_{i=1}^{k-1}2^{n_i}
    \le\sum_{i=1}^{n_{k-1}}2^i < 2^{n_{k-1}+1} \le \epsilon_k/2.
    $$  
    Therefore, $\abs{H\cap S_k}$ is the unique element of $N_k$ closest to
    $\abs{H\cap S}$.
  \end{proof}

  We now have all the ingredients to state and prove our theorem
  about the strongly cospectral subgroup of the cubelike graph
  $\Cay(G,S)$ introduced at the beginning of this subsection.
 
  \begin{theorem} Let $\Cay(G, S)$ be the graph defined at the beginning of
    subsection~\ref{qgraphs}. Suppose that the dimensions $\dim V_i=n_i$ satisfy the condition that for all $i=2$,..,$k$, we have $n_i>2n_{i-1}+7$.
    Then the order of the  strongly cospectral subgroup of
    $\Cay(G,S)=\square_{i=1}^k\Cay(V_i,S_i)$ is at least $2^k$.
    There exist arbitrarily large sets of mutually strongly cospectral
  vertices in cubelike graphs.
  \end{theorem}
  \begin{proof} By Lemma~\ref{hyperplane_cuts}(v), our hypothesis
    implies that the condition \eqref{ebound} of Lemma~\ref{capdetermine} is satisfied.
    Therefore if $H$ is either $G$ or a hyperplane of $G$,  $\abs{H\cap S}$  determines
    $\abs{H\cap S_i}$ for $i=1$,\dots, $k$, which means that each
    $\sigma_i=\sum_{s\in S_i}s$ belongs to the strongly cospectral subgroup, according to
    our discussion at the beginning of subsection~\ref{idea}.
    It follows from  Lemma~\ref{sumQ} that $\sigma_i=p_i$.  Since the $p_i$
    are linearly independent, the dimension of the strongly cospectral subgroup
    is at least $k$. The final statement holds because there is no restriction on the number $k$ of factors,  and the dimensions $n_i$ can be chosen recursively to satisfy the hypothesis.    
  \end{proof}

  \section{Concluding remarks}
  We shall use the same notation as in the Introduction.
  In the quantum walk on the graph $X$, we say that there is
  {\it pretty good state transfer} (PGST) from vertex $a$ to vertex $b$ if for all
  $\epsilon>0$, there exists $t>0$ such that $\abs{U(t)_{b,a}}\geq 1-\epsilon$.
  It is well known ([Lemma 13.1]\cite{G12}, due to D. Morris) that in order to
  have PGST from $a$ to $b$, the two vertices must be strongly cospectral.
  It can also  be shown that the existence of PSGT between vertices
  is  an equivalence relation. Furthermore, in the case of Cayley graphs
  the equivalence class of the identity element is a subgroup, just as
  for the relation of strong cospectrality. We shall denote this
subgroup by $G_{pg}$ and the strongly cospectral subgroup by $G_{sc}$.
We know that the set of vertices in a Cayley graph  for which there is
perfect state transfer from the identity vertex  consists of  the identity vertex and at most one involution, so this set is a group $G_{pt}$ of order  at most 2. Thus we have for any Cayley graph,
  \begin{equation}
    G_{pt} \le G_{pg}\le G_{sc}.
  \end{equation}
  For graphs in general, Pal and Bhattacharjya \cite[Example 4.1]{PB} have given the example of $P_2\square P_3$, in which the PGST class has 4 elements.
  Since there is no absolute upper bound for $\abs{G_{sc}}$,  it is natural to  ask whether or not there is an absolute bound  for $\abs{G_{pg}}$.
  We have not determined $\abs{G_{pg}}$ for  the examples constructed in section~\ref{hetero}. As for the cubelike graphs in section~\ref{cubelike},
  Hermie Monterde has pointed out to me that for periodic graphs (which includes graphs, such as cubelike graphs, whose eigenvalues are integers) PGST is equivalent to perfect
  state transfer \cite[Proposition 1.4]{Pal}.
  Necessary and sufficient conditions for PGST are  given in \cite[Theorem 2]{BCGS}.
  
  \section*{Acknowledgements}
  I would like to thank Soffia Arnadottir and Chris Godsil for some fruitful
  discussions.  Soffia  also helped with some early computer calculations.
  Thanks also to Ada Chan from whom I first learned of the question on the size
  of a strong cospectrality class, and to Hermie Monterde for her helpful comments
  on an earlier version of this paper. Finally, I thank the referees for their
  thoughtful suggestions to improve the exposition.


\end{document}